\documentclass{amsart}

\usepackage{amsmath, amsfonts,amssymb}
\usepackage{graphicx}
\usepackage{url}
\usepackage[myheadings]{fullpage}

\newcommand{\bN}{\mathbb{N}}
\newcommand{\bR}{\mathbb{R}}
\newcommand{\bP}{\mathbb{P}}

\newcommand{\bE}{\mathbb{E}}

\newtheorem{definition}{Definition}[section]
\newtheorem{corollary}{Corollary}[section]

\newtheorem{theorem}{Theorem}[section]
\newtheorem{lemma}{Lemma}[section]
\newtheorem{remark}{Remark}[section]

\begin{document}

\bibliographystyle{amsplain}

\author{Hiroya Hashimoto}
\address[H.~Hashimoto]{Sanwa Kagaku Kenkyusho Co., Ltd.}
\email{hiro\_hashimoto@skk-net.com}

\author{Takahiro Tsuchiya}
\address[T.~Tsuchiya]{School of Computer Science and Engineering, The University of Aizu}
\email[Corresponding author]{suci@probab.com}

\title{
Convergence rate of 
stability problems of SDEs 
with (dis-)continuous coefficients
}



\maketitle

\begin{abstract}
We consider the stability problems of one dimensional SDEs 
when the diffusion coefficients satisfy the so called Nakao-Le Gall condition.  
%
The explicit rate of convergence of the stability problems 
are given by the Yamada-Watanabe method without the drifts.  
We also discuss the convergence rate for the SDEs driven by the symmetric $\alpha$ stable process. 
These stability rate problems are extended to the case where 
the drift coefficients are bounded and in $L^1$. 
It is shown that the convergence rate is invariant under 
the removal of drift method for the SDEs driven by the Wiener process. 
\end{abstract}

\section{Introduction} 
Consider the following sequence of 
one-dimensional stochastic differential equations, SDEs for short, 
\begin{equation}\label{sde:stability}
X_n (t) = X_n (0)+\int_{0}^t b_n (X_n (s)) ds+\int_{0}^t \sigma_n (X_n (s)) dW_s, 
\end{equation}
and consider the solution  $X$ given by 
\begin{equation}\label{sde0}
X(t) = X(0)+\int_{0}^t b (X(s)) ds+\int_{0}^t \sigma (X(s)) dW_s, 
\end{equation}
where $\{ W_s \}_{s \geq 0 }$ is a Wiener process and 
$b_n : \bR \rightarrow \bR$ and $\sigma_n : \bR \rightarrow \bR$ for $n \in \bN $ are coefficients 
which tend to $b$ and $\sigma$  respectively in some sense, as $n \rightarrow \infty$. 
The convergence of the sequence $\{ X_n \}_{n \in \bN}$ to $X$ 
named as stability problems was introduced by Stroock and Varadhan in \cite{SV1979} 
to solve the martingale problems for unbounded coefficients $b$ and $\sigma$.   
The stability problem in the {\it strong} sense was treated by Kawabata-Yamada \cite{KY} 
in the case where the diffusion coefficients are H\"older continuous of exponent $\alpha \geq 1/2$. 
Le Gall  investigated  the stability problems,  
when the diffusion coefficients are positive and squared finite quadratic variation in \cite{LeGall1983}.  
Kaneko and Nakao showed in \cite{KN1988} that 
the pathwise uniqueness implies the stability property 
if the coefficients satisfy some condition on the modulus continuity  
or if the diffusion coefficient is positive definite. 

The rate of convergence of the Euler-Maruyama scheme 
to the solution has been discussed by Deelstra and Delbaen in \cite{Del1998} 
and their results are considerably generalized  by Gy\"ongy and R\'asonyi in \cite{GR2011}. 
In their investigations, the Yamada-Watanabe method introduced in \cite{WY}, \cite{YW} and \cite{Y1978} plays essential roles. 
Gy\"ongy and R\'asonyi  in \cite{GR2011} obtained the rate of convergence 
in the case where the diffusion coefficients are 
$(1/2 +  \gamma)$-H\"older continuous with the suitable drift coefficients in $L^1$. 
It is remarkable that the rate of convergence is given by $n^{-\gamma}$ where $ \gamma > 0$ 
and also is given by $(\log n)^{-1}$ in the case where $\gamma=0$. 

These results suggest that the rate of convergence of the stability problems may depend also on the 
modulus continuity or irregularity of diffusion coefficients. 
Since the coefficient may be discontinuous 
under the Nakao-Le Gall condition, 
it seems to be very interesting to investigate the rate of convergence of the stability problems. 

The present paper is organized as follows: 
Section \ref{sec:notation} is devoted to the assumption and preliminaries. 
Our main result is shown in Section \ref{stability problem}. 
Let us consider first drift-less system 
\begin{align}
X_n (t)&=X_n (0)+\int_{0}^t \sigma_n (X_n (s)) dW_s, \label{sde without drifts}  \\
X     (t)&=X     (0)   +\int_{0}^t \sigma (X(s)) dW_s, \label{sde without drifts 2}
\end{align}
for $t \geq 0$ and $X_n (0) \equiv X (0)$. 
These diffusion coefficients satisfy the so called Nakao-Le Gall condition 
(See Definition $\ref{conditionC}$ below).  
Assume that there exists a positive constant $C_0$ such that 
\begin{align*} 
\int_{\bR} | \sigma_n (x) - \sigma (x)  | dx \leq C_0 n^{-1} 
{\rm \ or \ }
\sup_{x \in\bR} | \sigma_n (x) - \sigma (x)  |  \leq C_0 n^{-1},
\end{align*}
for all $\in \bN$. 
It will be shown that there exist positive constants $C_p \ (p \geq 1)$ such that 
\begin{align*}
\bE \left[ \left| X (t) - X_n (t) \right| \right] \leq C_1 (\log n)^{-\frac{1}{2}} \ ( 0 \leq t \leq T )
\end{align*}
and
\begin{align*}
\bE \left[ \sup_{0 \leq t \leq T} \left| X (t) - X_n (t) \right|^p \right] \leq C_p (\log n)^{-\frac{p}{4(p+1)}} 
\end{align*}
holds for $p > 1$ and $n>2$. 
In Section \ref{section:sas}, 
we also discuss 
the convergence rate of the stability problems 
in the case where the SDEs are driven by 
a symmetric $\alpha$ stable process without the drift for $\alpha \in (1,2)$. 
Finally, the stability problems with the drift coefficients 
will be discussed using the removal of drift method 
in Section \ref{Removal invariant}.

%
%
%
%
\section{Assumptions and Preliminaries}\label{sec:notation}
On a probability space $(\Omega, \mathcal{F}, \bP )$ 
equipped with the filtration $\{ \mathcal{F}_t \}_{t \geq 0}$ satisfying the usual condition, 
let us consider the stochastic differential equations $(\ref{sde:stability})$ and $(\ref{sde0})$. 
For the simplify, we assume that $X (0) \equiv X_n (0) $. 
The following  condition 
on the diffusion coefficients was proposed by Nakao \cite{N1972} 
and modified by Le Gall \cite{LeGall1983}.   
\begin{definition}\label{conditionC}
In this paper, we say that a real valued function $\sigma$ satisfies the Nakao-Le~Gall condition and write $\sigma \in \mathcal{C}_{NL} (\epsilon, f)$
if $\sigma$ satisfies the following statements:   
\begin{enumerate}
\item There exists a positive real number $\epsilon$ such that 
\begin{equation*}
\epsilon  \leq \sigma (x) 
\end{equation*}
holds for any $x$ in $\mathbb{R}$. 
\item There exists an increasing function $f$ such that 
\begin{align*}
|\sigma (x) - \sigma (y) |^2 \leq | f(x) - f(y) |
\end{align*}\label{conditionC2}
holds for every $x$ and $y$ in $\mathbb{R}$. 
\end{enumerate}
In addition, 
if the function $f$ is bounded on an interval $I \subset \bR$ as follows: 
\begin{align*}
\| f \|_{I, \infty} :=\sup_{x \in I} |f(x)| < \infty, 
\end{align*}
then we denote 
$\sigma \in  \mathcal{C}_{NL} (\epsilon, \| f \|_{I, \infty})$. 
In particular, in the case where $I=\bR$,  
we write $\sigma \in  \mathcal{C}_{NL} (\epsilon, \| f \|_{\infty})$ 
and $\| f \|_{\infty} :=\| f \|_{\bR, \infty}$. 
\end{definition}

\begin{remark}
For the equation $(\ref{sde0})$, 
if $b$ is bounded measurable function and 
$\sigma \in \mathcal{C}_{NL} (\epsilon, \| f \|_{\infty})$, 
then there exists a unique strong solution $X$, 
see Theorem $1.3~(3)$ of the paper \cite{LeGall1983}. 
\end{remark}

\begin{remark}\label{rem:disc}
If the Nakao-Le~Gall condition holds,  $\sigma \in \mathcal{C}_{NL} (\epsilon, f )$, 
we can select a sequence of smooth functions $\{ f_l \}_{l \in \mathbb{N}}$ such that 
for continuous points $x, y \in \bR$,  
\[
|\sigma (x) - \sigma (y) |^2 \leq | f_l (x) - f_l (y) | 
{\rm  \ and \ } f_l (x) \uparrow f (x) \ (l \rightarrow \infty) .
\]
Then for an  progressively measurable process $X$ 
we have that
\[
\int_{0}^{t} f(X(s)) ds = \lim_{l \to \infty} \int_{0}^{t} f_l (X(s)) ds
\]
holds almost surely. 
We use this fact in the proof of Theorem \ref{Thm:main1} 
through \ref{Thm:main3}. 
\end{remark}

Now, we shall introduce the Yamada-Watanabe method. 
Take a decreasing sequence $( a_m )_{ m \in \bN }$ 
of positive numbers 
satisfying $\infty > a_1 > \cdots > a_m > \cdots \downarrow 0$ and 
$\int_{a_{m}}^{a_{m-1}} x^{-1} \, dx=m $. 
Let us consider a smooth symmetric around the origin function $\varphi_{m} (\cdot)$ 
with support in $(-a_{m-1}, -a_{m})$ and $(a_{m}, a_{m-1})$ such that 
$0 \leq \varphi_{m} (x) \leq {2} \left( x m \right)^{-1}$ holds 
and $\int_{\bR}\varphi_{m} (y) dy=1 $. 
For all $x \in \bR$ define 
\[
u_m (x) = \int_{0}^{|x|} \int_0^y \varphi_m (z) \, dz \, dy.
\]
Then we have 
\begin{align}\label{Def:Yfunc}
|x| \geq u_m (x) \geq  -a_{m-1} + |x|
, \ \ 1 \geq | u'_m (x) | 
\end{align}
and $u''_m(x) =\varphi_m (x)$ holds for $x\in \bR$, see \cite{Y1978}. 

\section{A strong convergence rate of the stability problems}\label{stability problem}
In this section, we consider the convergence rate of the stability problems 
for the drift-less system 
$(\ref{sde without drifts})$ and $(\ref{sde without drifts 2})$ under the Nakao-Le Gall condition. 
The stability problem with the $L^1$-convergence 
was discussed by Le Gall \cite{LeGall1983}. 
In more detail, 
the explicit rate of the convergence in $L^p$ sup-norm is obtained as follows: 
\begin{theorem}\label{Thm:main1}
Let $T>0$ and $p \geq 1$ 
and let  
$X_n$ for $n \in \bN$ be solutions of $(\ref{sde without drifts})$ 
and 
$X$ be a solution of $(\ref{sde without drifts 2})$  such that  
 $\bE \left| X (0) \right|^p < \infty $ and $X_n (0) \equiv X (0)$. 
Suppose that $\sigma$ and $\{ \sigma_n \}_{n \in \bN}$ satisfy the Nakao-Le~Gall condition, i.e.,   
$\sigma, \sigma_n \in \mathcal{C}_{NL} (\epsilon, \| f\|_{\infty} )$. 
Assume one of the following stability rate conditions holds: 
there exists a positive constant $C_0$ such that 
\begin{align}\label{con:L1} 
\int_{\bR} | \sigma_n (x) - \sigma (x)  |^{2} dx \leq C_0 n^{-1},  \  n \in \mathbb{N}, 
\end{align}
or 
\begin{align}\label{con:uni} 
\sup_{x \in \bR} | \sigma_n (x) - \sigma (x)  |  \leq C_0 n^{-1},  \  n \in \mathbb{N}. 
\end{align}
Then there exist positive constants $C_p \ (p \geq 1)$ such that 
\begin{align*}
\bE \left[ \left| X (t) - X_n (t) \right| \right] \leq C_1 (\log n)^{-\frac{1}{2}} \ ( 0 \leq t \leq T )
\end{align*}
and
\begin{align}\label{th2:eq2}
\bE \left[ \sup_{0 \leq t \leq T} \left| X (t) - X_n (t) \right|^p \right] \leq C_p (\log n)^{-\frac{p}{4(p+1)}} 
\end{align}
holds for $p > 1$ and $n>2$. 
\end{theorem}

%
%
\begin{proof}
The boundedness of $\sigma$ and $\sigma_n$ implies that 
\begin{align}\label{Lp-property}
\bE \left[ \sup_{0 \leq t \leq T} |X(t)|^p \right] + \sup_{n \in \bN}\bE \left[ \sup_{0 \leq t \leq T} |X_n (t)|^p  \right] 
< \infty. 
\end{align}
Define 
\begin{align*}
Y_n (t) :=  X (t) - X_n (t) . 
\end{align*}

Let $u_m$ be a function as defined in Section \ref{sec:notation}. 
By $(\ref{Def:Yfunc})$ and then applying 
It\^o's  formula to $u_m \left(Y_n ( \cdot ) \right) $ 
\begin{subequations}
\begin{align}\label{eq:9a}
| X (t) - X_n (t)  | &\leq a_{m-1} +u_m (Y_n (t) ) \\ 
			  &=a_{m-1} +M (t) +J(t), \notag
\end{align}
where $M$ and $J$ are defined as follows:  
\begin{align*}
 M  (t)  &=  \int_0^t u'_m (Y_n (s)) 
 	\left\{  \sigma \left(X (s) \right) -  \sigma_n ( X_n (s) )   \right\} dW_s , \\ 
 J  (t)  &=  \frac{1}{2} \int_0^t u''_m (Y_n (s))  
 	\left| \sigma \left(X (s) \right)  -  \sigma_n ( X_n (s) ) \right|^2 ds . 
\end{align*}
Now we also define $J^\sigma$ and $J^Y$, 
\begin{align}\label{eq:9b}
 J  (t)  &\leq  \int_0^t u''_m (Y_n (s))  \left| \sigma \left(X (s) \right)  - \sigma_n ( X (s) )  \right|^2 ds  
 \\ &  \qquad   
 +   \int_0^t u''_m (Y_n (s))   \left| \sigma_n \left(X (s) \right)  - \sigma_n ( X_n (s) )  \right|^2 ds ,\notag 
  \\ &=: J^\sigma (t)  +  J^Y (t) , \  {\rm  say \ for \  } 0 \leq t \leq T \notag. 
\end{align}
Note that by $(\ref{eq:9b})$
\begin{align}\label{eq:9f}
| X (t) - X_n (t)  |^p  
	&\leq \left| a_{m-1} + | M (t) |  +J^\sigma (t)  +  J^Y (t) \right|^p \\
	&\leq 4^{p-1} \left\{ a_{m-1}^p  +|M (t)|^p + \left| J^\sigma (t) \right|^p  + | J^Y (t)|^p \right\}. \notag 
\end{align}
Then taking the sup-norm over the time interval $[0,T]$, 
we have 
\begin{align}\label{eq:9f}
&\sup_{0 \leq t \leq T} | X (t) - X_n (t)  |^p \\
	&\leq 4^{p-1} \left\{ a_{m-1}^p 
	+ \sup_{0 \leq t \leq T}|M (t)|^p +\left| J^\sigma (T) \right|^p   + | J^Y (T)|^p \right\}. \notag 
\end{align}
Selecting the smooth functions $ f_l \  ( l \in \bN )$ in Remark \ref{rem:disc}, we have for $x, y \in \bR$, 
\[ {f_l (x) - f_l (y)}= (x-y )\int_0^1 f_l^{'} (x + \theta (y-x)) d\theta . \]
Then we define and consider $ J^{Y}_l$ as follows: 
\begin{align*}
 J^{Y}_l (t) 
  &:= \int_0^t u''_m (Y_n (s))   \left| f_l \left(X (s) \right)  - f_l ( X_n (s) )  \right| ds
\\&= \int_0^t \varphi_m (Y_n (s))  |Y_n(s)| 1_{\left( 0<|Y_n(s)|  \right)} \frac{\left|  f_l \left(X (s) \right) -  f_l \left(X_n  (s)  \right)  \right| }{|Y_n(s)|} ds    \notag
  \\ 
  &\leq
 {2}{m}^{-1}    \int_0^t  \int_0^1  f_l^{'} ( Z^{\theta}(s)) d\theta ds 
\end{align*} 
holds where 
for $\theta \in [0,1]$ we define
\[ Z(t) \equiv Z^{\theta}(t):= X (t)  + \theta \left( X_n ( t ) -X (t) \right). 
\]
The fact that the martingale part of $Z$ is expressed by 
\[
\int_0^{t} \left\{ (1-\theta) \sigma \left( X ( s) \right) +\theta  \sigma_n \left( X_n ( s) \right)  \right\}  dW_s  =: \int_0^t \tilde{\sigma} (s)dW_s, 
\]
and $(1-\theta)\sigma(x) +\theta \sigma_n (x) \geq \epsilon$ for $x \in \bR$ implies that 
we have $\langle Z, Z \rangle_t  \geq \epsilon^2 t $. 
Now the occupation times formula implies that we have 
\begin{align*}
\int_0^t  \int_0^1   f_l^{'}  ( Z^{\theta }(s) ) d\theta ds  
&\leq  \epsilon^{-2}  \int_0^1  \int_{-\infty}^{\infty}   f_l^{'}(a)  L_t^a (Z^{\theta}) da d\theta 
\end{align*}
where $L_t^a (Z)$ stands for the local time of $Z$ accumulated at $a$ until time $t$. 
By the Meyer-Tanaka formula and $L^p$ integrability of the solutions $X$ and $X_n$ of $(\ref{Lp-property})$, we have 
\begin{align}\label{eq:9h}
c_{L} := \sup_{\theta \in [0,1]} \sup_{a \in \bR}  \bE [ \left(  L_T^a (Z^{\theta}) \right)^p ]  < \infty  . 
\end{align}  
Then we obtain 
\begin{align*}
\bE [| J^Y_l (t)|^p ] 
	&\leq
	2^p (m\epsilon^2)^{-p}  \bE [ \left(   \int_0^1  \int_{-\infty}^{\infty}   f_l^{'}(a)  L_t^a (Z^{\theta}) da d\theta  \right)^p   ] 
	\\&\leq 
	2^p (m\epsilon^2)^{-p} 
	\left(   \int_0^1  \int_{-\infty}^{\infty}   f_l^{'}(a)  da d\theta  \right)^{p-1}   
	\bE [   \int_0^1  \int_{-\infty}^{\infty}   f_l^{'}(a)  \left(  L_t^a (Z^{\theta})  \right)^p da d\theta   ] 
	\\&\leq 
	2^p (m\epsilon^2)^{-p} 
	\| f_l \|^{p-1}_{\infty}   
	\int_0^1  \int_{-\infty}^{\infty}   f_l^{'}(a)  \bE [  \left(  L_t^a (Z^{\theta})  \right)^p ] da d\theta    
	\\&\leq 
	2^p (m\epsilon^2)^{-p}  \| f_l \|^{p}_{\infty}  c_L.  
\end{align*}
By the monotone convergence theorem, we have $ J^{Y}_l \rightarrow  J^{Y}$ as $l \rightarrow \infty$. 
the $L^p$ estimate for $J^Y (T)$ in $(\ref{eq:9b})$,  
\begin{align*}
\bE \left| J^Y (T) \right|^p &\leq 2^p (m\epsilon^2)^{-p} \| f \|^{p}_{\infty}  c_L . 
\end{align*}  

On the other hand, let us consider $J^\sigma$ in $(\ref{eq:9b})$. 
By the construction of $\varphi_m$ we have
\begin{align*}
J^\sigma (t) 
&= \int_0^t u''_m (Y_n (s))  \left| \sigma \left(X (s) \right)  - \sigma_n ( X (s) )  \right|^2 ds 
\\ &\leq 2 (m a_m )^{-1}  \int_0^t  \left| \sigma \left(X (s) \right)  - \sigma_n ( X (s) )  \right|^2 ds. 
\end{align*}

If the convergence rate $(\ref{con:L1})$ holds, 
then using H\"older's inequality we obtain 
\begin{align*}
&\left( \int_0^t  \left| \sigma \left(X (s) \right)  - \sigma_n ( X (s) )  \right|^2 ds \right)^p 
\\ &\leq \epsilon^{-2p}\left( \int_{\bR}  \left| \sigma \left(a \right)  - \sigma_n ( a )  \right|  L_t^a (X) da \right)^{p}
\\ &\leq \epsilon^{-2p}\left( \int_{\bR}  \left| \sigma \left(a \right)  - \sigma_n ( a )  \right| da \right)^{p-1}
	\int_{\bR}  \left| \sigma \left(a \right)  - \sigma_n ( a )  \right| | L_t^a (X) |^p da
\\ &\leq \epsilon^{-2p}\left( C_0 n^{-1} \right)^{p-1}
	\int_{\bR}  \left| \sigma \left(a \right)  - \sigma_n ( a )  \right| | L_t^a (X) |^p da,
\end{align*}
and then we obtain 
\begin{align*}
\bE \left( \int_0^t  \left| \sigma \left(X (s) \right)  - \sigma_n ( X (s) )  \right|^2 ds \right)^p 
\leq \epsilon^{-2p}\left( C_0 n^{-1} \right)^{p-1} \times  C_0 n^{-1}  c_L
= \epsilon^{-2p}c_L \left( C_0 n^{-1} \right)^{p} ,
\end{align*}
where $L_t^a (X)$ is the local time of $X$ such that  
\begin{align*}
 \sup_{a \in \bR}  \bE [ \left( L_t^a (Z^{0}) \right)^p ] =  \sup_{a \in \bR} \bE [ \left( L_t^a (X)\right)^p ] \leq c_{L}. 
\end{align*}

On the other hand, the rate of convergence of $(\ref{con:uni})$ implies that 
\begin{align*}
&\left( \int_0^t  \left| \sigma \left(X (s) \right)  - \sigma_n ( X (s) )  \right|^2 ds \right)^p \leq (t C_0 n^{-1})^p. 
\end{align*}

Then, under the stability rate condition  $(\ref{con:L1})$ or $(\ref{con:uni})$,  
the term of $J^\sigma  $ is estimated in $L^p$ as follows: 
\begin{align*}
\bE [|J^\sigma (T) |^p ] 
	&\leq  2^p (m a_m )^{-p} \left( \epsilon^{-2p}  c_{L} +T^p  \right)(C_0 n^{-1})^{p}. 
\end{align*}
In short, combining these estimates with $(\ref{eq:9b})$
we obtain 
\begin{align}\label{eq:9e}  
\bE \left| J (T) \right|^p \leq  2^p {m}^{-p}  c_{JY}   +  2^p (m a_m n )^{-p}  c_{J\sigma} , 
\end{align}  
where we define 
\begin{align*}  
c_{JY} &:= \epsilon^{-2p}  \| f \|^{p}_{\infty}  c_L  
\\ 
c_{J\sigma} &:=  \left( \epsilon^{-2p}  c_{L} +T^p  \right) C_0^{p}. 
\end{align*}

Now let us estimate $Y_n$ in the case of $p=1$, 
\begin{align*}
| Y_n (t)  | 
	&\leq  a_{m-1} +M (t)  + J^\sigma (t)  +  J^Y (t). 
\end{align*}	
Then we obtain 
\begin{align}\label{eq:9c}
\bE \left[ \left| Y_n (t) \right| \right] & \leq a_{m-1}+
2c_{JY} {m}^{-1}   +  2c_{J\sigma} (m a_m n )^{-1}  =  A_{m,n}, \ {\rm \ say. }
\end{align}
%
Here, 
in the Yamada-Watanabe method 
let us choose $( a_m )_{m \in \bN  }$ as $a_m = \exp \left( -m(m+1)/2 \right), \ a_0=1$  
and select a sequence $(m_n)_{n \in \bN}$ such that 
\begin{align}\label{eq:9g}
\frac{1}{ a_{m_n} n}  \leq 1 
\end{align}
holds for $n > 2$. 
Since we have that $ a_{m-1} \leq {1}/{m}$ for any $m \in \bN$, 
there exists a finite positive number $c_{a}$ such that 
$A_{m_n,n}$ in $(\ref{eq:9c})$ is bounded by  
\begin{align}\label{eq:9d}
\bE \left[ \left| Y_n (t) \right| \right]  \leq A_{m_n,n} \leq   c_{a}    ( \log n)^{-\frac{1}{2}}. 
\end{align}
\end{subequations}

%
%
We shall obtain the  $L^p$-estimate $(\ref{th2:eq2})$ 
from this $L^1$-estimate $(\ref{eq:9d})$. 
Let us estimate the quadratic process $\langle M \rangle$, 
\begin{align*}
\langle M \rangle_T 
&\leq \int_0^t \left|  \sigma \left(X (s) \right) -  \sigma_n \left(X_n (s) \right)   \right|^2 ds \\
&\leq
\int_0^T  \left|  \sigma \left(X (s) \right) -  \sigma_n \left(X (s) \right)   \right|^2 ds 
+
\int_0^T  \left|  \sigma_n \left(X (s) \right) -  \sigma_n \left(X_n (s) \right)   \right|^2 ds 
\\
&=: \langle M^\sigma \rangle_T  + \langle M^Y \rangle_T, \ {\rm say}.
\end{align*}
By the same argument as we estimate $\bE | J(T) |^p $ in $(\ref{eq:9e})$ with $c_{J\sigma}$,  
we have 
\begin{align*}
\bE [\langle M^\sigma \rangle_T^{\frac{p}{2}} ] 
\leq \sqrt{ \bE [\langle M^\sigma \rangle_T^{{p}} ] }
&\leq   \sqrt{ c_{J\sigma}  n^{-p} }. 
\end{align*}
On the other hand, for any positive number $y$ we have 
\begin{align*}
&\langle M^Y \rangle_T
=
\int_0^T  \left|  \sigma_n \left(X (s) \right) -  \sigma_n \left(X_n (s) \right)   \right|^{2} ds  
\\&
\leq \int_0^T  \left| f_l \left(X (s) \right) -  f_l \left(X_n (s) \right)   \right| 1_{(|Y_n (s)| >y)}ds 
\\&
+ \int_0^T  \left| f_l \left(X (s) \right) -  f_l \left(X_n (s) \right)     \right| 1_{(|Y_n (s) | \leq y)}ds 
\\
&\leq
2 \| f \|_{\infty}  \int_0^T  1_{(|Y_n (s)| >y)} ds 
+
y \int_0^T   \int_0^1   f_l' (Z^\theta (s) )   d\theta ds .
\end{align*}
Therefore, putting $y=\left( \log n \right)^{-\frac{1}{2(p+1)}}$, 
there exists a positive number $c_{M}$ such that 
\begin{align*}
&\bE [\langle M^Y \rangle_t^{\frac{p}{2}}] \leq \sqrt{ \bE [\langle M^Y \rangle_t^{p}] }
\\&\leq 
\sqrt{  2^{p} \max \left( (2 \| f \|_{\infty})^{p}  T^{p}  y^{-1} ( \log n )^{-\frac{1}{2}}, \  
y^{p}  { c_{JY} } \right)}
\leq c_{M} \left( \log n \right)^{-\frac{p}{4(p+1)}}
. 
\end{align*}
Then 
by combining these estimates with  $(\ref{eq:9f})$ 
and the Burkholder-Davis-Gundy Inequality 
we observe that 
there exists a constant value $c_p$ such that 
\begin{align*}
&\bE [ \sup_{0 \leq t \leq T} | X (t) - X_n (t)  |^p  ]
	\\&\leq c_p \left\{ a_{m-1}^p 
	+\bE [\langle M^{\sigma} \rangle_t^{\frac{p}{2}}] +\bE [\langle M^Y \rangle_t^{\frac{p}{2}}] 
	+\bE [  \left| J^\sigma (T) \right|^p ]   + \bE [ | J^Y (T)|^p] \right\}. 
	\\&\leq c_p \left\{ a_{m-1}^p 
	+\sqrt{c_{J\sigma}  n^{-p}} +\bE [\langle M^Y \rangle_t^{\frac{p}{2}}] 
	+2^p c_{JY} {m}^{-p}   +  2^p c_{J\sigma} (m a_m n )^{-p}  \right\}. 
\end{align*}
Since we have selected the sequence $(m_n)_{n \in \bN}$ as in  $(\ref{eq:9g})$, 
\begin{align*}
&\bE [ \sup_{0 \leq t \leq T} | X (t) - X_n (t)  |^p  ]
	\\&\leq c_p (\log n)^{-\frac{p}{4(p+1)}} \left\{ 
	1
	+ \sqrt{c_{J\sigma}}   +c_{M}
	+2^p c_{JY}   +  2^p c_{J\sigma}\right\}, 
\end{align*}
for $n>2$. 
Therefore we obtain the desired result. 
\end{proof}

\section{SDEs driven by symmetric $\alpha$-stable processes}\label{section:sas} 
In this section, 
we introduce a symmetric $\alpha$ stable process $Z$ and 
the sequence of SDEs driven by $Z$: 
\begin{align}
X(t) &= X(0) + \int_0^t \sigma (X(s-)) dZ_s ,\label{sde:sas1}  \\
X_n (t) &= X_n (0) + \int_0^t \sigma_n (X_n (s-)) dZ_s\ \ ( n \in \bN ), 
\end{align}
where 
both $\{ \sigma_n \}_{n \in \bN}$ and $\sigma$ satisfy 
the following Belfadli-Ouknine condition. 
\begin{definition}
We say that 
a function $\sigma$ satisfies the Belfadli-Ouknine condition if 
$\sigma$ satisfies 
Definition $\ref{conditionC}$ $(1)$ and the modified condition $(2)'$ of Definition $\ref{conditionC}$: 
there exists an increasing function $f$ such that 
\begin{align*}
| \sigma (x)- \sigma (y) |^{\alpha} &\leq | f(x) - f(y)| {\rm \ and \ }  \| f \|_{\infty} < \infty
\end{align*}
holds for any $x, y \in \bR$. 
\end{definition}
Under the assumptions, 
the pathwise uniqueness holds for $X$ and also $X_n$, $n \in \bN$, which is proven by Belfadli and Ouknine in \cite{BO2008}. 
Moreover, it is shown by Hashimoto \cite{Hashimoto2013} 
that the solution is realized by the stability problem . 
The local time is again the key to compute the strong convergence rates via stability problem. 
For the results of the local time for the symmetric $\alpha$ stable process see 
K. Yamada \cite{KYamada2002} and Salminen and Yor \cite{YorSalminen2007}. 
\begin{theorem}\label{Thm:main2}
Let $\alpha$ be a positive number with $1< \alpha <2$. 
Assume $X_n (0) \equiv X (0)$ with $\bE \left| X_n (0) \right|^{\alpha -1} < \infty $. 
If $\sigma$ and $\{ \sigma_n \}_{n \in \bN}$ satisfy the Belfadli-Ouknine condition 
and either of the stability rate condition $(\ref{con:L1})$ or $(\ref{con:uni})$, 
then there exists a positive constant $C_2$ 
such that 
\begin{align}
\bE \left[ \left| X (t) - X_n (t) \right|^{\alpha-1} \right] \leq  C_2  (\log n)^{-\frac{\alpha-1}{2}}
\end{align}
holds for  $0 \leq t \leq T$ and $n>2$. 
\end{theorem}
\begin{proof}
Using \'Emery's inequality in \cite{Emery1978},  we have 
\[
 \bE \left[ \sup_{0 \leq t \leq T} |X(t)|^{\alpha -1} \right] + \sup_{n \in \bN}\bE \left[ \sup_{0 \leq t \leq T} |X_n (t)|^{\alpha -1} \right] 
<\infty. 
\] 
Define $Y_n (t) :=  X (t) - X_n (t) $. 
Let $\varphi_m $ be a function as defined in Section \ref{sec:notation}. 
Define 
\[v_m = v \ast \varphi_m,  \]
 where $v(x)=|x|^{\alpha-1}$ and $\ast$ stands for the convolution. 
Then we have 
\[
-(a_{m-1})^{\alpha -1} +v_m (x) \leq 
v(x) \leq  (a_{m-1})^{\alpha -1} +v_m (x) , 
\]
holds for $x \in \bR$, see \cite{Hashimoto2013}. 

By the It\^o formula we have 
\begin{align*}
| Y_n(t) |^{\alpha-1} &\leq (a_{m-1})^{\alpha-1} +v_m (Y_n (t) ) \\
			  &=(a_{m-1})^{\alpha-1} +M (t) +J(t),  
\end{align*}
where 
$M$ and $J$ are defined as follows:  
\begin{align*}
 M  (t)  &=  \int_0^t u'_m (Y_n (s)) \left\{  \sigma \left(X (s) \right) -  \sigma_n ( X_n (s) )   \right\} dZ_s , \\
 J  (t)  &=   K_\alpha \int_0^t \varphi_m (Y_n (s))  \left| \sigma \left(X (s) \right)  -  \sigma_n ( X_n (s) ) \right|^\alpha ds,
\end{align*} 
where we define a constant $K_{\alpha}=  -  \Gamma(\alpha) \cos ({\alpha \pi}/{2})/2 \geq 0$ for $\alpha \in (1,2)$ 
from the gamma function $\Gamma (\cdot )$. 

Now we also define $J^\sigma$ and $J^Y$ such that  
\begin{align*}
 J  (t)  &\leq  
2 K_\alpha \int_0^t \varphi_m (Y_n (s))   \left\{ \sigma \left(X (s) \right)  -  \sigma_n ( X (s) ) \right\}^\alpha ds  
 \\ &  \qquad    +
2 K_\alpha \int_0^t \varphi_m (Y_n (s))    \left| \sigma_n \left(X (s) \right)  -  \sigma_n ( X_n (s) ) \right|^\alpha ds , 
  \\ &=: 2 K_\alpha \left( J^\sigma (t)  +  J^Y (t) \right) , \  {\rm  say} . 
\end{align*}

For $\theta \in [0,1]$, let us define
\[ V(t) \equiv V_t \equiv V^{\theta}(t):= X (t)  + \theta \left( X_n ( t ) -X (t) \right). 
\]
Notice that we have for $x, y \in \bR$, 
\[ {f_l (x) - f_l (y)}= (x-y )\int_0^1 f_l^{'} (x + \theta (y-x)) d\theta . \]
Then we have that 
\begin{align*}
 J^{Y}_l (t) 
 &:=\int_0^t \varphi_m (Y_n (s))  \left|  f_l \left(X (s) \right) -  f_l \left(X_n  (s) \right)  \right| ds    \notag
  \\ 
 &=
\int_0^t \varphi_m (Y_n (s))  |Y_n(s)| 1_{\left( 0<|Y_n(s)|  \right)} \frac{\left|  f_l \left(X (s) \right) -  f_l \left(X_n  (s)  \right)  \right| }{|Y_n(s)|} ds    \notag
  \\ 
  &\leq
 {2}{m}^{-1}    \int_0^t  \int_0^1  f_l^{'} ( V^{\theta}(s)) d\theta ds .
\end{align*}

Define 
\[
A_t := \int_0^{t} |H_s|^{\alpha}  ds 
{\rm \ where \ }
H_s :=  \sigma \left( X ( s) \right) +\theta \left(  \sigma_n \left( X_n ( s) \right)-  \sigma \left( X ( s) \right) \right) 
\]
By the time-change method described in \cite{BO2008} or \cite{Tsuchiya2007}, we have
\begin{align*}
\int_0^t  f_l^{'} ( V_s ) ds 
=\int_0^t    f_l^{'} ( \tilde{Z}_{A_s}) |H_s|^{-\alpha} dA_s 
\leq \epsilon^{-\alpha} \int_0^t     f_l^{'} ( \tilde{Z}_{A_s})  dA_s ,
\end{align*}
where $V_s =: \tilde{Z}_{A_s}$. Making the random change of variable 
$s=\tau_u :=\inf \{ t \geq 0: A_t >u \}$, we obtain  
\begin{align*}
\int_0^t  f_l^{'} ( V^{\theta}(s)) ds 
\leq \epsilon^{-\alpha} \int_0^{A_t}    f_l^{'} ( \tilde{Z}_u )du 
\leq \epsilon^{-\alpha} \int_0^{ \| \sigma \|_{\infty}^{\alpha} T}    f_l^{'} ( \tilde{Z}_u )du, 
\end{align*}
where $\tilde{Z}$ is the symmetric $\alpha$ stable process 
with respect to the filtration $\mathcal{G}_t :=\mathcal{F}_{\tau^{-1}(t) }$. 
By the occupation time formula of the local time $ L_t^a $ of $\tilde{Z}$ we have 
\begin{align*}
\int_0^{\tilde{T} }    f_l^{'} ( \tilde{Z}_u )du
=
\int_{-\infty}^{\infty}     f_l^{'} (a) L_{\tilde{T}}^a ( \tilde{Z} ) da, 
\end{align*}
for $\tilde{T}:=\| \sigma \|_{\infty}^{\alpha} T$, for an example see \cite{Bertoinb1996}. 
Indeed, by the Tanaka-Meyer-K.~Yamada formula in \cite{KYamada2002} implies that 
\begin{align*}
L_t^a (\tilde{Z}) 
=  |\tilde{Z}_t -a|^{\alpha -1}- |a|^{\alpha-1} -N_t^a
&\leq \left| |\tilde{Z}_t -a|^{\alpha -1}- |a|^{\alpha-1} \right| -N_t^a 
\\ &\leq \left| |\tilde{Z}_t -a|- |a| \right|^{\alpha-1} -N_t^a
\\ &\leq | \tilde{Z}_t  |^{\alpha-1} -N_t^a, 
\end{align*}
where $N^a$  is a squared martingale such that    
\begin{align*}
\langle N^a \rangle_t 
=c_\alpha \int_0^t \frac{ds}{| \tilde{Z}_s -a |^{2-\alpha}}. 
\end{align*}
Then we obtain  
\[
c_{\tilde{L}}:=\sup_{a \in \bR} \bE[ L_{\tilde{T}}^a (\tilde{Z}) ] < \infty. 
\]
By Jensen's inequality for the concave function $x^{\alpha-1}$, we have  
\begin{align*}
\bE [| J^Y_l (t)|^{\alpha -1} ] &\leq 
	2^{\alpha -1} (m \epsilon^\alpha)^{ -(\alpha -1)}
	\bE [ \left( \int_0^1 \int_{-\infty}^{\infty}     f_l^{'} (a) L_{\tilde{T}}^a ( \tilde{Z} ) da d\theta \right)^{\alpha-1}   ] 
	\\ &\leq 
	2^{\alpha -1}  (m \epsilon^\alpha)^{ -(\alpha -1)}
	 \left( \int_0^1 \int_{-\infty}^{\infty}     f_l^{'} (a)  \bE [  L_{\tilde{T}}^a ( \tilde{Z} )] da d\theta   \right)^{\alpha-1}   
	\\ &\leq 
	2^{\alpha -1}  (m \epsilon^\alpha)^{ -(\alpha -1)}
	 \left( \|f \|_{\infty} c_{\tilde{L}} \right)^{\alpha-1} . 
\end{align*}
On the other hand, let us consider $J^\sigma$, 
\begin{align*}
J^\sigma (t) 
&= \int_0^t  \varphi_m (Y_n (s))  \left| \sigma \left(X (s) \right)  - \sigma_n ( X (s) )  \right|^\alpha ds 
\\ &\leq 2 (m a_m )^{-1}  \int_0^t  \left| \sigma \left(X (s) \right)  - \sigma_n ( X (s) )  \right|^\alpha ds. 
\end{align*}
Note that 
\begin{align*}
 \int_0^t  \left| \sigma \left(X (s) \right)  - \sigma_n ( X (s) )  \right|^\alpha ds
	\leq  \| \sigma_0 \|_{\infty}^{\alpha-1}  \epsilon^{-\alpha} 
	\int_{\bR}  \left| \sigma \left(a \right)  - \sigma_n ( a )  \right| L_t^a (X) ds, 
\end{align*}
where $L_t^a (X)$ is the local time of $X$ such that 
\[  \sup_{a \in \bR}  \bE [  L_t^a (Z^{0}) ] =  \sup_{a \in \bR} \bE [L_t^a (X) ] \leq c_{\tilde{L}}, \]
with $c_{\tilde{L}}$ as in $(\ref{eq:9h})$ in Section \ref{stability problem}, 
and we define 
\[
\| \sigma_0 \|_{\infty} := \sup_{x \in \bR} (| \sigma (x)| , |\sigma_n (x)| ) < \infty.
\]
The stability rate condition $(\ref{con:L1})$ or $(\ref{con:uni})$
implies that 
\[
\bE [|J^\sigma (t) |^{\alpha -1} ] \leq  2^{\alpha -1} (m a_m )^{-(\alpha -1)} 
\left( 
t^{\alpha-1}
+
\left( \|f \|_{\infty} c_{\tilde{L}} \right)^{\alpha-1} 
\right)\left(
 \| \sigma_0 \|_{\infty}
C_0 n^{-1} \right)^{\alpha -1}.
\]

In short, we obtain 
\begin{align*}  
\bE \left| J (t) \right|^{\alpha-1} \leq  
	(4 K_\alpha )^{\alpha-1} c_{JY} {m}^{-(\alpha-1)}   +  (4 K_\alpha )^{\alpha-1} 
	c_{J\sigma} (m a_m n)^{-(\alpha-1)}. 
\end{align*}  
where we define 
\begin{align*}  
c_{JY} &:= \left(  \epsilon^{-\alpha}  \| f \|_{\infty} c_{\tilde{L}} \right)^{\alpha-1}, \\ 
c_{J\sigma} &:=   
\left( t^{\alpha-1} + \left( \|f \|_{\infty} c_{\tilde{L}} \right)^{\alpha-1}  \right)
\left( \| \sigma_0 \|_{\infty} C_0  \right)^{\alpha -1}
\end{align*}

Now let us estimate $|Y|^{\alpha-1}$,  
\begin{align*}
| Y_n (t)  |^{\alpha-1} 
	&\leq  (a_{m-1})^{\alpha-1} +M (t)  + J^\sigma (t)  +  J^Y (t). 
\end{align*}	
Then we obtain 
\begin{align*}
\bE \left[ | Y_n (t)  |^{\alpha-1}  \right] & \leq (a_{m-1})^{\alpha-1} +
(4 K_\alpha )^{\alpha-1}\left( c_{JY} {m}^{-(\alpha-1)}   +  c_{J\sigma} (m a_m n )^{-(\alpha-1)}  \right)
=  A_{m,n}, {\rm \ say. }
\end{align*}
%
Here let us consider the one of the sequence $( a_m )_{m \in \bN  }$ such that $a_m = \exp \left( -m(m+1)/2 \right)$. 
Then 
select a sequence $(m_n)_{n \in \bN}$ such that 
\[
\frac{1}{ a_m n}  \leq 1 
\]
holds for for $n > 2$. 
Since we have that $ a_{m-1} \leq {1}/{m}$ for any $m \in \bN$, 
there exists a finite positive number $c_a$ such that 
\begin{equation*}
\bE \left[ |Y_n(t) |^{\alpha-1}  \right] \leq 
A_{m_n,n} \leq   c_a    ( \log n)^{-\frac{\alpha-1}{2}}. 
\end{equation*}
\end{proof}

\section{Invariant property under removal drift}\label{Removal invariant}
In this section, we consider the SDEs of $(\ref{sde0})$ with the drifts of SDEs 
driven by the Wiener process $W_s$.  
Now we introduce the method of removal of the drifts. 

Suppose that there exists a strong solution $X_n$ and $X$ to the equation 
$(\ref{sde:stability})$ and $(\ref{sde0})$ and taking values on $I \equiv (l, k)$ for $-\infty \leq l < k \leq \infty$.  
Consider the coefficients $b, b_n$ and $\sigma, \sigma_n $ for $n \in \bN$ such that the functions 
$ b_n  \sigma^{-2}_n$ and  $b  \sigma^{-2}$ belong to $L^1(I)$, i.e., 
\begin{align}\label{ass:cff1}
\sup_{n \in \bN}\int_{I} | b_n (u) \sigma^{-2}_n (u) | du < \infty  {\rm \ and \ }
\int_{I} | b (u) \sigma^{-2} (u) du | < \infty. 
\end{align}
Now let us consider the scale functions given by 
\begin{align}\label{Def: F and Fn}
s_n' (x) := \exp \left( -2 \int^x \frac{b_n (u) }{\sigma_n^2 (u)}  du\right)   {\rm \ and \ }  
s'     (x) :=\exp \left( -2 \int^x \frac{b     (u) }{\sigma^2     (u)}  du\right).  
\end{align}
By It\^o's formula we have the drift-less processes 
$\bar{X}_n (t) := s_n (X_n (t))$ taking values on $s_n (I)$  
and $\bar{X} (t) := s (X(t))$ taking values on $s (I)$  such that 
\begin{align}\label{sde:without the drifts term}
\bar{X}_n (t) - \bar{X}_n (0)  &= \int_0^t \bar{\sigma} (\bar{X}_n (s) ) dW_s  \\
\bar{X} (t)     - \bar{X} (0)      &= \int_0^t \bar{\sigma}_n (\bar{X} (s) ) dW_s, 
\end{align}
where 
\begin{align*}
\bar{\sigma}_n (\bar{x}): = (\sigma_n s'_n ) \circ s_n^{-1} (\bar{x}) = \sigma_n (s^{-1}_n (x) ) s'_n (s^{-1}_n (\bar{x}) )  &{\rm \ for \ }  \bar{x} \in s_n (I) ,\\
\bar{\sigma}     (x)        := (\sigma s' ) \circ s^{-1} (x) = \sigma      (s^{-1}    (x) ) s'    (s^{-1}      (x) )  &{\rm \ for \ }  \bar{x} \in s (I). 
\end{align*}

The Nakao-Le Gall condition is invariant under the removal of drifts in the following sense. 
\begin{lemma}\label{remark: invariant}
If the diffusion coefficient $\sigma$ satisfies the Nakao-Le Gall's condition, 
$\sigma \in \mathcal{C}_{NL} (\epsilon, \| f \|_{\infty})$, and 
the drift coefficient $b$ is bounded and in $L^1$, then so does $\bar{\sigma}$  on $s(I)$. 
\end{lemma}
\begin{proof}
By the assumption $\sigma \in \mathcal{C}_{NL} (\epsilon, \| f \|_{\infty})$,  we have 
\[
\int_I | b(u) | |\sigma^{-2} (u) | du \leq \epsilon^{-2} \int_I | b(u) |   du < \infty. 
\]
In addition, the boundedness of $b$ implies that for $x \in I$ 
\[
\exp \left(- 2 \| b \sigma^{-2}\|_{L^1}  \right)  \leq 
s' (x) \leq \exp \left( 2 \| b \sigma^{-2}\|_{L^1}  \right) . 
\]
Then there exists a positive number $\epsilon'$ such that 
\[
0 < \epsilon' \leq \bar{\sigma} (\bar{x})
\]
holds for $\bar{x} \in s(I)$. 

On the other hand, using the boundedness of $b$ we have  
\[
 |s'' (x)| = |\left( -2 {b (x) }{\sigma^{-2} (x)} \right)  s' (x) | 
 	 \leq 2 \| b \sigma \|_{I, \infty} \exp \left( 2 \| b \sigma^{-2}\|_{L^1}  \right) <\infty 
\]
where $ \| b \sigma \|_{I, \infty} :=\sup_{u \in I} |b (u) \sigma (u) | $. 
Thus $s'$ is Lipschitz continuous. 
Therefore there exists a monotone function $f_b$ such that 
\[
| \bar{\sigma} (\bar{x}) - \bar{\sigma} (\bar{y}) | \leq | f_b (x)- f_b (y) |,
\]
holds for $x,y \in s(I)$. 
\end{proof}

This result suggests that the stability problem is independent of the drift coefficient itself.  
To be more precious, 
these convergence rates are also invariant under the removal drift as follows.  
\begin{corollary}[Invariant property]\label{corollary:Ip}
Suppose that 
the diffusion coefficient $\sigma, \sigma_n $ for $n \in \bN$ satisfy the Nakao-Le Gall's condition, 
$\sigma, \sigma_n  \in \mathcal{C}_{NL} (\epsilon, \| f \|_{\infty})$, and 
the drift coefficient $b$ and $b_n$ are uniformly bounded and in $L^1$. 
If there exists a constant $C_0$ such that 
\begin{align*} 
\int_{x \in \bR} | b_n (x)- b(x)  | + | \sigma_n (x) - \sigma (x)  | dx \leq C_0 n^{-1},
\end{align*}
then there exists a positive number $\bar{C_0}$ such that 
\[
\int_{S_n} | \bar{\sigma} (\bar{x}) -  \bar{\sigma}_n (\bar{x}) |  d\bar{x} \leq \bar{C_0}n^{-1}. 
\]
where $S_n = s_n (I) \cap s (I) $. 
\end{corollary}
\begin{proof}
Let us observe 
\begin{align*}
& \left| \bar{\sigma}_n (\bar{x}) - \bar{\sigma} (\bar{x}) \right| 
	\\ & \leq
|\sigma_n (s^{-1}_n (\bar{x}) ) |
	\left(  |s'_n (s^{-1}_n (\bar{x}) )  - s' (s^{-1}_n    (\bar{x}) ) | + | s' (s^{-1}_n (\bar{x}) ) -  s' (s^{-1} (\bar{x}) ) |  \right) 
	\\ & \ \  +
	|s' (s^{-1} (\bar{x}))| 
	\left( | \sigma_n (s^{-1}_n (\bar{x}) ) -  \sigma (s^{-1}_n (\bar{x}) ) | +| \sigma (s^{-1}_n (\bar{x}) )  - \sigma  (s^{-1}    (\bar{x}) )  |\right) 
	\\ & \leq
	2 \| f  \|_{\infty}
	 \left(  |s'_n (s^{-1}_n (\bar{x}) )  - s' (s^{-1}_n    (\bar{x}) ) | + | s' (s^{-1}_n (\bar{x}) ) -  s' (s^{-1}(\bar{x}) ) |  \right)
	\\ & \ \  +
	\| s' ( s^{-1})  \|_{\infty} \left( | \sigma_n (s^{-1}_n (\bar{x}) ) -  \sigma (s^{-1}_n (\bar{x}) ) | +| \sigma (s^{-1}_n (\bar{x}) )  - \sigma      (s^{-1}    (\bar{x}) )  |\right)
\end{align*}
By the definition of the scale functions $s_n$ and $s$ we have 
\begin{align*}
&s(x) - s_n (x) = \int^{x} s' (y) - s_n' (y) dy
\\&= \int^{x} \exp \left (-2  \int^{y} b(u) \sigma^{-2} (u) du \right)  dy- \exp \left (-2  \int^{y}  b_n (u) \sigma^{-2}_n (u) du \right)  dy 
\\&= \int^{x} \left(  \log s'(y)- \log s_n' (y) \right)
\left(\int_0^1 \exp \left( \theta \log s'(y) + (1-\theta )\log s_n' (y) \right) d\theta  \right) 
dy 
\\&= \int^{x} \left(  \log \frac{ s'(y)}{  s_n' (y) } \right)
\left(\int_0^1 {\left( s'(y) \right)}^\theta  \left( s_n' (y)\right)^{1-\theta } d\theta \right)  dy .
\end{align*}
Now put $c_{F1}$ and $c_{F2}$ as follows: 
\begin{align*}
c_{s1} &:= \min \left\{ \exp \left(- 2 \sup_{n \in \bN} \|  b_n \sigma_n^{-2} \|_{L^1}  \right),  \exp \left(- 2 \| b \sigma^{-2}\|_{L^1}  \right) \right\} , 
\\
c_{s2} &:=
\max \left\{ \exp \left( 2  \sup_{n \in \bN}\|  b_n \sigma_n^{-2} \|_{L^1}  \right),  \exp \left(2 \| b \sigma^{-2}\|_{L^1}  \right) \right\} . 
\end{align*}
By the assumption, we have 
\begin{align*}
\left| \log \frac{ s'(y)}{  s_n' (y) } \right|
	&\leq
	\int_{I} | b_n (u) \sigma^2 (u)- \sigma_n^2 (u) b (u)  |   \sigma^{-2} (u)  \sigma_n^{-2} (u)  du 
	\\& \leq
	\int_{I}  | b_n (u)- b (u) | \sigma^{-2}_n (u)  
	+    |\sigma^2 (u) - \sigma_n^2 (u)| | b (u) |   \sigma^{-2} (u)  \sigma_n^{-2} (u)  du .
\end{align*}
Then by the rate of the stability convergence rate condition we have 
\begin{align*}
\sup_{x \in \bR}| s'(x) - s'_n(x) |  &\leq c_{s2} c_{b,\sigma} n^{-1},
\end{align*}
where $c_{b,\sigma}: =\| \sup_{n \in \bN}   n  \log ( s'  /  s_n' )\|_{L^1} < \infty$. 

For $\bar{x} = s^{-1} (y) \left( y \in s(I) \right)$ and 
$\bar{x} = s_n^{-1} (y_n ) \left( y_n  \in s_n (I) \right)$ we have 
\begin{align*}
(s^{-1})'(\bar{x}) - (s_n^{-1})'(\bar{x})= \frac{1}{s' (y)} - \frac{1}{s'_n (y_n)} =\frac{ s'_n (y_n) - s' (y)}{s' (y) s'_n (y_n)}. 
\end{align*}
Then it implies that 
\begin{align*}
\sup_{\bar{x} \in \bR} | (s^{-1})' (\bar{x}) - (s_n^{-1})'(\bar{x}) | \leq c_{s2} {c_{s1}^{-2} }c_{b,\sigma} n^{-1}=:c'_{b,\sigma}n^{-1}. 
\end{align*}

Since we have 
\[
s' (s^{-1}_n (\bar{x}) ) -  s' (s^{-1} (\bar{x}) )   = ( s^{-1}_n (\bar{x}) -  s^{-1} (\bar{x}) ) 
\left( \int_0^1 \left( -2 {b}{\sigma^{-2}} s' \right) (\theta s^{-1}_n (\bar{x}) +(1-\theta)s^{-1} (\bar{x})) d\theta \right),
\]
then we obtain that 
\begin{align*}
&\int_{S_n} \left| s' (s^{-1}_n (\bar{x}) ) -  s' (s^{-1} (\bar{x}) )  \right| d \bar{x} 
\\&=\int_{S_n} \left|  s^{-1}_n (\bar{x}) -  s^{-1} (\bar{x}) \right| 
\left| \int_0^1 \left( -2 {b}{\sigma^{-2}} s' \right) (\theta s^{-1}_n (\bar{x}) +(1-\theta)s^{-1} (\bar{x})) d\theta \right| d \bar{x}
	\\&\leq 
	2 n^{-1} c'_{b,\sigma}
	\times \int_0^1 \int_{S_n}  | \left({b}{\sigma^{-2}} \right) (\theta s^{-1}_n (\bar{x}) +(1-\theta)s^{-1} (\bar{x})) | d\bar{x} d\theta
	\\&\leq 
	2 n^{-1} c'_{b,\sigma}
	\times \int_0^1 \int_{S_n}  | \left({b}{\sigma^{-2}} \right) (g_\theta (\bar{x})) | g'_{\theta} (\bar{x}) d\bar{x} d\theta
	\\&\leq 2 n^{-1} c'_{b,\sigma} \| {b}{\sigma^{-2}}  \|_{L^1(I)} < \infty
\end{align*}
where we define 
\[
g_\theta (\bar{x}) :=\theta s^{-1}_n (\bar{x})  +(1-\theta) (s^{-1}    (\bar{x}) )  
\]
and note that $ c_{s2}^{-1} \leq g'_{\theta} (\bar{x}) \leq c_{s1}^{-1}$. 

Considering the transformation of variable $y_n = s_n^{-1} (\bar{x})$, 
we have 
\begin{align*}
&\int_{S_n}  | \sigma_n (s^{-1}_n (\bar{x}) ) -  \sigma (s^{-1}_n (\bar{x}) ) | d \bar{x}
\leq c_{s2}
\int_{S_n}  | \sigma_n (y_n ) -  \sigma (y_n) | d y_n \leq c_{s2} C_0 n^{-1} .  
\end{align*}

Using the sequence of $\{ f_l \}_{l \in \bN} $ in Remark \ref{rem:disc},  we have 
\begin{align*}
&\int_{S_n} | \sigma (s^{-1}_n ( \bar{x} ) )  - \sigma      (s^{-1}    ( \bar{x} ) )  |^2 d\bar{x} 
\leq \int_{S_n} | f_l (s^{-1}_n (\bar{x}) )  -f_l (s^{-1}    (\bar{x}) )  | dx 
	\\&=\int_{S_n} |s^{-1}_n (\bar{x})  - s^{-1} (\bar{x}) | \left| \int_0^1  f'_l (\theta s^{-1}_n (\bar{x})  +(1-\theta) (s^{-1}    (\bar{x}) )  ) d \theta  \right| d\bar{x} 
	\\&\leq c'_{b,\sigma}  n^{-1} 
		 \int_0^1 \int_{S_n}
		   c_{s2} g'_\theta (\bar{x}) 
		  f'_l ( g_\theta (\bar{x})  ) d\bar{x} d \theta . 
	\\&\leq c'_{b,\sigma}  n^{-1}  \| f \|_{\infty}.
\end{align*}

Therefore we obtain 
\begin{align*}\label{scale invariant}
&\int_{S_n}\left| \bar{\sigma}_n (\bar{x}) - \bar{\sigma} (\bar{x}) \right| d \bar{x}
	\\ & \leq
	\left( \|\sigma_n  \|_{S_n, \infty} 
	\left( c_{s2} c_{b,\sigma}  +2  c'_{b,\sigma}\| {b}{\sigma^{-2}}  \|_{L^1(I)} \right)
	+
	c_{s2} 
	(  c_{s2} C_0 + c'_{b,\sigma}   \| f \|_{S_n, \infty } ) \right) 
	 n^{-1},
\end{align*}
and hence we obtain the invariant property. 
\end{proof}

In short, we obtain the extended results of Theorem \ref{Thm:main1} with the drifts. 
\begin{theorem}\label{Thm:main3}
Let $T>0$ and $p \geq 1$ and let $X_n$ for $n \in \bN$ be solutions of $(\ref{sde:stability})$ 
and $X$ be a solution of $(\ref{sde0})$ such that  
$\bE \left| X (0) \right|^p < \infty $ and $X_n (0) \equiv X (0)$. 
Suppose that $\sigma$ and $\{ \sigma_n \}_{n \in \bN}$ satisfy the Nakao-Le~Gall condition 
$\sigma, \sigma_n \in \mathcal{C}_{NL} (\epsilon, \| f\|_{\infty} )$. 

If the stability rate conditions holds, i.e. 
there exists a positive constant $C_0$ such that 
\begin{align*} 
\int_{x \in \bR} | b_n (x)- b(x)  | + | \sigma_n (x) - \sigma (x)  | dx \leq C_0 n^{-1}, 
\end{align*}
Then there exist positive constants $C_p \ (p \geq 1)$ such that 
\begin{align*}
\bE \left[ \left| X (t) - X_n (t) \right| \right] \leq C_1 (\log n)^{-\frac{1}{2}} \ ( 0 \leq t \leq T )
\end{align*}
and
\begin{align*}
\bE \left[ \sup_{0 \leq t \leq T} \left| X (t) - X_n (t) \right|^p \right] \leq C_p (\log n)^{-\frac{p}{4(p+1)}} 
\end{align*}
holds for $p > 1$ and $n>2$. 
\end{theorem}

\begin{remark}
Assume $\bE [\sup_{ 0\leq t \leq T}|X (t)|^2] < \infty$. 
This implies that $\bE [ \int_{\bR} L_T^a (X) da  ] = \bE [\langle X \rangle]_T < \infty$, 
and then the same convergence rate is obtained under the stability rate of $(\ref{con:L1})$. 
\end{remark}

\providecommand{\bysame}{\leavevmode\hbox to3em{\hrulefill}\thinspace}
\providecommand{\MR}{\relax\ifhmode\unskip\space\fi MR }
\providecommand{\MRhref}[2]{%
  \href{http://www.ams.org/mathscinet-getitem?mr=#1}{#2}
}
\providecommand{\href}[2]{#2}

\end{document}